\documentclass[12pt,a4paper,reqno]{amsart}
\usepackage{amssymb}
\usepackage{amscd}
\usepackage{enumerate}
\usepackage{graphicx}
\usepackage{siunitx}
\usepackage{bm}
\numberwithin{equation}{section}

\usepackage[pagebackref]{hyperref}

\usepackage{mathtools}
\usepackage[tableposition=top]{caption}
\usepackage{booktabs,dcolumn}

\DeclareFontFamily{OT1}{rsfs}{}
\DeclareFontShape{OT1}{rsfs}{n}{it}{<-> rsfs10}{}
\DeclareMathAlphabet{\mathscr}{OT1}{rsfs}{n}{it}

\theoremstyle{plain}

\newtheorem{theorem}{Theorem}[section]
\newtheorem{proposition}[theorem]{Proposition}
\newtheorem{lemma}[theorem]{Lemma}

\theoremstyle{definition}

\newtheorem{remark}[theorem]{Remark}

\newcommand\E{\mathbb{E}}

\newcommand\R{\mathbb{R}}
\newcommand\Z{\mathbb{Z}}

\begin{document}

\title[Optimal sine and sawtooth inequalities]{Optimal
sine and sawtooth inequalities }

\author{Louis Esser}
\address{UCLA Department of Mathematics, Los Angeles, CA 90095-1555.}
\email{esserl@math.ucla.edu}

\author{Terence Tao}
\address{UCLA Department of Mathematics, Los Angeles, CA 90095-1555.}
\email{tao@math.ucla.edu}

\author{Burt Totaro}
\address{UCLA Department of Mathematics, Los Angeles, CA 90095-1555.}
\email{totaro@math.ucla.edu}

\author{Chengxi Wang}
\address{UCLA Department of Mathematics, Los Angeles, CA 90095-1555.}
\email{chwang@math.ucla.edu}



\begin{abstract}
We determine the optimal inequality
of the form $\sum_{k=1}^m a_k\sin kx\leq 1$, in the sense
that $\sum_{k=1}^m a_k$ is maximal.
We also solve exactly the analogous problem
for the sawtooth (or signed fractional part) function.
Equivalently, we solve
exactly an optimization problem about equidistribution on the unit circle.
\end{abstract}

\maketitle


In this paper, we determine the optimal inequality
of the form
$$\sum_{k=1}^m a_k\sin kx \leq 1$$
for each positive integer $m$, in the sense that $\sum_{k=1}^m a_k$
is maximal (Theorem \ref{sine}). Namely, $\sum_{k=1}^m a_k$ is on the order
of $\log m$, and we compute it exactly. This is a natural extremal problem
in Fourier analysis.

We also solve the analogous optimization problem
for the {\it sawtooth }(or {\it signed
fractional part}) function
$g(x)=x+\bigl\lfloor \frac{1}{2}-x\bigr\rfloor$,
which takes values in $(-1/2,1/2]$. Namely, we find
an optimal inequality of the form
$$\sum_{k=1}^m b_k\, g(kx) \leq 1$$
for each positive integer $m$, in the sense that
$\sum_{k=1}^mb_k$ is maximal (Theorem \ref{sawtooth}).
See the figures
in sections \ref{sectionsawtooth} and \ref{sectionsine}
for examples of these inequalities,
which show striking cancellation among dilated sine
or sawtooth functions.

By linear programming duality, these inequalities are equivalent
to statements about equidistribution on the unit circle.
In particular: for each positive integer $m$
and every probability measure on the real line,
at least one of the dilated sawtooth functions $g(kx)$
for $k\in \{1,\ldots,m\}$
must have small expected value,
and we determine the optimal bound in terms of $m$ (Theorem
\ref{sawtooth}). It is on the order of $1/\log m$, and
we compute it exactly.

These results were motivated by an application
to algebraic geometry. For smooth complex
projective varieties of general type, the {\it volume }is a positive
rational number that measures the asymptotic growth
of the plurigenera $h^0(X,mK_X)$.
Before the authors' series
of papers in 2021, the varieties
of general type with smallest known volume in high dimensions $n$
were those found by Ballico, Pignatelli, and Tasin,
with volume roughly $1/n^n$ \cite{BPT}.

Using our equidistribution result for the sawtooth function,
for any constant $c<1$, version 1 of this paper on the arXiv
constructed varieties
of general type in all sufficiently high dimensions $n$
with volume less than $1/e^{c n^{3/2}(\log n)^{1/2}}$.
The equidistribution
result was used to optimize the constant $c$.

Three of the authors then went further by different methods,
finding $n$-dimensional varieties of general type
with volume less than $1/2^{2^{n/2}}$ \cite{ETW}.
In view of that improvement,
we have omitted the algebro-geometric
application from this paper. There should be other ways
to apply our optimization results
for the sine and sawtooth functions.

\subsection{Acknowledgments}

LE and BT were supported by NSF grant DMS-2054553.  TT was supported by a Simons Investigator grant, the James and Carol Collins Chair, the Mathematical Analysis \& Application Research Fund Endowment, and by NSF grant DMS-1764034.
Thanks to John Ottem, Sam Payne, and Miles Reid for their suggestions.


\section{Dilated fractional parts of a random real number}
\label{sectionsawtooth}

In this section, we prove an optimal inequality for the sawtooth
function (Theorem \ref{sawtooth}). By linear programming duality,
this is equivalent to an optimal bound in a problem
about equidistribution on the unit circle.

For a real number $x$, let $\lfloor x \rfloor$ denote the greatest integer less than or equal to $x$ and $\lceil x \rceil$ the smallest integer greater than or equal to $x$.  Note that
\begin{equation}\label{lf}
 \lfloor x+n \rfloor = \lfloor x\rfloor + n \;\text{ and }\;
\lceil x+n \rceil = \lceil x\rceil + n
\end{equation}
for any integer $n$, and that
\begin{equation}\label{cc}
\lceil x \rceil = - \lfloor -x\rfloor.
\end{equation}
We also define the lower fractional part
\begin{equation}\label{lfrac}
\{x\} \coloneqq x - \lfloor x \rfloor
\end{equation}
which takes values in $[0,1)$, and the upper fractional part
\begin{equation}\label{ufrac}
\{x\}^* \coloneqq x - \lceil x \rceil + 1
\end{equation}
which takes values in $(0,1]$.  Finally, define the signed fractional part
\begin{equation}\label{g-def}
g(x) \coloneqq x - \bigl\lceil x-\tfrac{1}{2} \bigr\rceil = x + \bigl\lfloor \tfrac{1}{2} - x \bigr\rfloor
\end{equation}
which takes values in $(-1/2,1/2]$. We call $g(x)$
the \emph{sawtooth function}.

For a (Borel) probability measure $\mu$ on the reals
and a positive integer $k$, define the expectation
$$ \E_\mu g(k x) \coloneqq \int_\R g(kx)\ d\mu(x).$$
Consider the quantity
\begin{equation}\label{gk}
 \min_{1 \leq k \leq m} \E_\mu g(k x),
\end{equation}
where $m$ is a natural number and $\mu$ is a probability measure on the reals.
Since each function $g(kx)$ is pointwise bounded by $1/2$, we trivially have the bound
\begin{equation}\label{minr}
 \min_{1 \leq k \leq m} \E_\mu g(k x) \leq \tfrac{1}{2},
\end{equation}
but one expects to do better as $m$ gets large.  For instance, from the Dirichlet approximation theorem one sees that
$$ \min_{1 \leq k \leq m} |g(kx)| \leq \frac{1}{m+1},$$
but this does not directly allow one to improve the bound \eqref{minr} since one cannot interchange the minimum and the expectation.  As it turns out, there is an improvement in $m$, and the optimal value of \eqref{gk} can be computed exactly, but it only decays like $\frac{1}{\log m}$ rather than $\frac{1}{m}$ as $m \to \infty$.

As a first attempt to control the quantity \eqref{gk}, one could try to estimate it by its unweighted mean
$$ \frac{1}{m} \sum_{k=1}^m \E_\mu g(k x).$$
However, this quantity can be quite large: in particular, if $\mu$ is the Dirac mass at $1/2$, then this mean is equal to $\frac{\lceil m/2 \rceil}{2m}$, which is asymptotic to $\frac{1}{4}$ as $m \to \infty$.  Closely related to this is the observation that the unweighted sum
$$ g(x) + g(2x) + \dots + g(mx)$$
of the $g(jx)$ can be much larger than $1$, and in particular equal to $\lceil m/2\rceil / 2$ when $x=1/2$.

However, one can obtain much better results by working with \emph{weighted} means of the $\E_\mu g(kx)$, or equivalently by \emph{weighted} linear combinations of the $g(jx)$; indeed by linear programming duality, we see that a bound of the form
$$  \min_{1 \leq k \leq m} \E_\mu g(k x) \leq \lambda$$
for all $\mu$ holds if and only if there exist non-negative coefficients $a_1,\dots,a_m$ with $\sum_{j=1}^m a_j \geq \frac{1}{\lambda}$ such that we have the dual inequality
$$ \sum_{k=1}^m a_k g(kx) \leq 1$$
for all $x$.  Thus to compute the minimal value of \eqref{gk}, we just need to find an optimal dual inequality.  

We begin with the model case where $m$ is a power of two, in which the dual inequality is particularly easy to establish.

\begin{proposition}
  Let $r$ be a natural number, and set $m = 2^r$.
\begin{itemize}
\item[(i)] For every real number $x$, we have
\begin{equation}\label{3r-easy}
2g(x) + \sum_{j=1}^r g(2^j x) \leq 1.
\end{equation}
\item[(ii)] We have
\begin{equation}\label{4r-easy}
\min_{1 \leq k \leq m} \E_\mu g(k x)  \leq \frac{1}{r+2}
\end{equation}
for every (Borel) probability measure $\mu$ on the real line. Moreover,
this is the optimal bound: equality is attained for the measure $\mu_m$
with mass $\frac{1}{r+2}$ at each of the numbers
$\frac{1}{2}, \frac{1}{4}, \frac{1}{8}, \ldots, \frac{1}{2^{r}}$,
and mass $\frac{2}{r+2}$ at $\frac{1}{2^{r+1}}$.
\end{itemize}
\end{proposition}

\begin{proof}
We begin with (i).  We observe the identity
\begin{equation}\label{tele}
 g(x) = \{2x\}^* - \{x\}^*
\end{equation}
for all real numbers $x$, since both sides of this equation are $1$-periodic, equal to $x$ on $(0,1/2]$, and equal to $x-1$ on $(1/2,1]$. Similarly,
$$ 2g(x) - \{2x\}^* = - 1_{\{x\}^* > 1/2},$$
where the indicator function $1_{\{x\}^* > 1/2}$ is defined to equal $1$ when $\{x\}^* > 1/2$ and vanish otherwise.
Thus we have the telescoping formula
\begin{equation}\label{telescoping}
 2g(x) + \sum_{j=1}^r g(2^j x) = \{2^{r+1} x\}^* - 1_{\{x\}^* > 1/2}
\end{equation}
(see Figure \ref{fig:combination2}.)  This establishes (i).

Now we prove (ii).  Integrating (i) against an arbitrary probability measure $\mu$ on $\R$, we conclude that
$$
2 \E_\mu g(x) + \sum_{j=1}^r \E_\mu g(2^j x) \leq 1.
$$
Thus we have
$$ \biggl(2 + \sum_{j=1}^r 1\biggr)
\min_{1 \leq k \leq m} \E_\mu g(k x)  \leq 1,$$
which gives the upper bound in \eqref{4r-easy}.

To establish the matching lower bound using the measure $\mu_m$, it suffices to show that
$$ \sum_{j=1}^{r} g\biggl( \frac{k}{2^j}\biggr) + 2 g\biggl( \frac{k}{2^{r+1}} \biggr) = 1$$
for all $k=1,\dots,m$.  But from \eqref{telescoping} with $x = k/2^{r+1}$, the left-hand side is equal to
$$ \{k\}^* - 1_{\{k/2^{r+1}\}^* > 1/2} = 1 - 0,$$
giving the claim.
\end{proof}

\begin{figure} [t]
\centering
\includegraphics[width=4in]{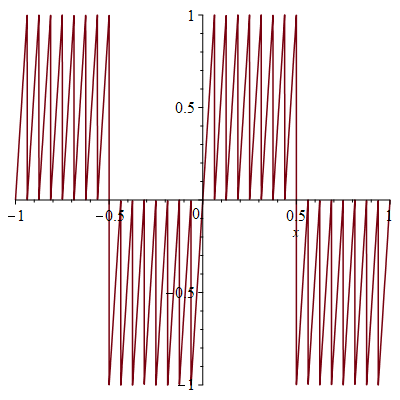}
\caption{The left-hand side of \eqref{3r-easy} when $m=8$.  We are showing that it is at most $1$.  Note the large number of locations where the bound of $1$ is attained, including the support $\{1/2,1/4,1/8,1/16\}$ of the optimal measure $\mu_8$.}
\label{fig:combination2}
\end{figure}

Now we handle the general case.

\begin{theorem}
\label{sawtooth}
  Let $r,m$ be natural numbers such that $2^r < m \leq 2^{r+1}$.  
\begin{itemize}
\item[(i)] For every real number $x$, we have
\begin{equation}\label{3r}
2g(x) + \sum_{j=1}^r g(2^j x) +  \sum_{\ell = 2^r+1}^m \frac{m}{\ell(\ell-1)}
\biggl( g(\ell x) + g((2^{r+1}+1-\ell)x)-g(x)\biggr) \leq 1.
\end{equation}
\item[(ii)] We have
\begin{equation}\label{4r}
\min_{1 \leq k \leq m} \E_\mu g(k x)  \leq \frac{2^r}{(r+1)2^r + m}
\end{equation}
for every (Borel) probability measure $\mu$ on the real line. Moreover,
this is the optimal bound: equality is attained for the measure $\mu_m$
with mass $\frac{2^r}{(r+1)2^r+m}$ at each of the numbers
$\frac{1}{2}, \frac{1}{4}, \frac{1}{8}, \ldots, \frac{1}{2^{r+1}}$
and mass $\frac{m}{(r+1)2^r+m}$ at $\frac{1}{2m}$.
\end{itemize}
\end{theorem}

In particular, the right side of (ii) is less than $\frac{\log 2}{\log m}$.
So (ii) says in particular: for every probability measure $\mu$
on the real line and every positive integer $m$, there is
a positive integer $k$ at most $m$ such that the expected value
$\E_{\mu} g(kx)$ is at most $\frac{\log 2}{\log m}$. This
is an equidistribution statement, sharpening
the rough idea that the image measure of $\mu$ on $\R/\Z$ 
under multiplication by
some not-too-large positive integer is not concentrated too much
in the first half of $[0,1]$.

It follows from our argument (in particular the properties
of the measure $\mu_m$ in (ii))
that statement (i) is an optimal inequality of the form
$\sum_{k=1}^m a_k g(kx)\leq 1$, in the sense that it has
the maximal value of $\sum a_k$ (namely, $r+1+\frac{m}{2^r}$)
among all inequalities of this form. See Remark \ref{non-unique}
on the non-uniqueness of this inequality.

\begin{figure} [t]
\centering
\includegraphics[width=4in]{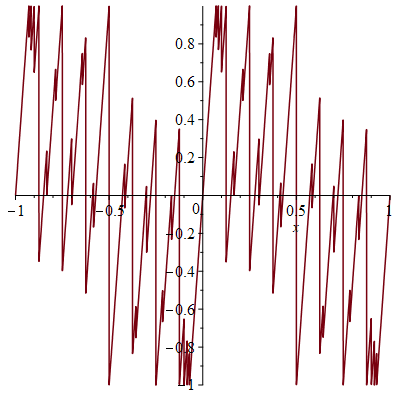}
\caption{The left-hand side of \eqref{3r} when $m=7$. Note the large number of locations where the bound of $1$ is attained, including the support $\{1/2,1/4,1/8,1/14\}$ of the optimal measure $\mu_7$.}
\label{fig:combination}
\end{figure}

\begin{proof}  
By \eqref{telescoping}, we can rearrange \eqref{3r} as
$$ \sum_{\ell = 2^r+1}^m \frac{m}{\ell(\ell-1)}
\biggl( g(\ell x) + g((2^{r+1}+1-\ell)x) - g(x)\biggr) \leq 1 - \{2^{r+1} x\}^* + 1_{\{x\}^* > 1/2}.$$
Both sides are $1$-periodic, so we may assume that $x \in (0,1]$, and thus we may write
\begin{equation}\label{jdef}
 x = \frac{j-\theta}{2^{r+1}}
\end{equation}
for some integer $1 \leq j \leq 2^{r+1}$ and some real number $0 \leq \theta < 1$.  In particular,
we have $\{2^{r+1} x\}^* = 1-\theta$.  From \eqref{g-def} we have
$$ g(\ell x) + g((2^{r+1}+1-\ell)x) - g(x)  = 2^{r+1} x + \bigl\lfloor \tfrac{1}{2} - \ell x \bigr\rfloor + \bigl\lfloor \tfrac{1}{2} + \ell x - 2^{r+1} x - x \bigr\rfloor - \bigl\lfloor \tfrac{1}{2} - x\bigr\rfloor.$$
But $-\lfloor \frac{1}{2} - x \rfloor = 1_{\{x\}^* > 1/2}$
since $x\in (0,1]$, so we have
$$ \sum_{\ell = 2^r+1}^m \frac{m}{\ell(\ell-1)} \bigl( -\bigl\lfloor \tfrac{1}{2} - x \bigr\rfloor \bigr) \leq 1_{\{x\}^* > 1/2},$$
since 
\begin{equation}\label{aum}
\sum_{\ell = 2^r+1}^m \frac{m}{\ell(\ell-1)} = m \biggl(\frac{1}{2^r}-\frac{1}{m}\biggr)
\leq 1.
\end{equation}
So it will suffice to establish the inequality
$$ \sum_{\ell = 2^r+1}^m \frac{m}{\ell(\ell-1)} \biggl( 2^{r+1} x + \bigl\lfloor \tfrac{1}{2} - \ell x \bigr\rfloor + \bigl\lfloor \tfrac{1}{2} + \ell x - 2^{r+1} x - x \bigr\rfloor \biggr) \leq \theta.$$
Writing $2^{r+1} x = j-\theta$, we can rearrange this using \eqref{lf} as
$$ \sum_{\ell = 2^r+1}^m \frac{m}{\ell(\ell-1)} \biggl( -\theta + \bigl\lfloor \tfrac{1}{2} - \ell x \bigr\rfloor + \bigl\lfloor \tfrac{1}{2} + \ell x + \theta - x \bigr\rfloor \biggr) \leq \theta.$$
By \eqref{aum}, it is equivalent
to show that
$$ \sum_{\ell = 2^r+1}^m \frac{m}{\ell(\ell-1)} \biggl( \bigl\lfloor \tfrac{1}{2} - \ell x \bigr\rfloor + \bigl\lfloor \tfrac{1}{2} + \ell x + \theta - x \bigr\rfloor\biggr) \leq \frac{m \theta}{2^r}.$$
We may cancel the factor of $m$ on both sides. The quantity $\lfloor \frac{1}{2} - \ell x \rfloor + \lfloor \frac{1}{2} + \ell x + \theta - x \rfloor$ is clearly an integer that is bounded above by
$$
\bigl(\tfrac{1}{2} - \ell x\bigr) + \bigl(\tfrac{1}{2} + \ell x + \theta - x\bigr) = 1 + \theta - x < 2,$$
and by \eqref{lfrac} it is equal to $1$ if and only if
$$ \bigl\{ \tfrac{1}{2} - \ell x \bigr\} + \bigl\{ \tfrac{1}{2} + \ell x + \theta - x \bigr\} = \theta - x,$$
so in particular
$$ 0 \leq \bigl\{ \tfrac{1}{2} - \ell x \bigr\} \leq \theta - x.$$
Thus it will suffice to show that 
\begin{equation}\label{test}
 \sum_{2^r < \ell \leq m: \; 0 \leq \{ \frac{1}{2} - \ell x \} \leq \theta - x} \frac{1}{\ell(\ell-1)} \leq \frac{\theta}{2^r}.
\end{equation}
We can write the left-hand side as
$$
 \sum_{k = 0}^\infty \; \sum_{2^r < \ell \leq m: \;
-k \leq \frac{1}{2} - \ell x \leq -k + \theta - x} \frac{1}{\ell(\ell-1)}
$$
or equivalently
\begin{equation}\label{k-1}
 \sum_{k = 0}^\infty \; \sum_{\max(2^r+1, \frac{k-\theta+1/2}{x}+1) \leq \ell \leq \min(\frac{k+1/2}{x}, m)} \frac{1}{\ell(\ell-1)}.
\end{equation}
Using the fundamental theorem of calculus to write $\frac{1}{\ell(\ell-1)} = \int_{\ell-1}^\ell \frac{dt}{t^2}$, we may upper bound this expression by
$$
 \sum_{k = 0}^\infty  \int_{[2^r,m] \cap [\frac{k-\theta+1/2}{x}, \frac{k+1/2}{x}]} \frac{dt}{t^2}
$$
and so (using $m \leq 2^{r+1}$) it will suffice to establish the bound
\begin{equation}\label{k-2}
 \sum_{k = 0}^\infty  \int_{[2^r,2^{r+1}] \cap [\frac{k-\theta+1/2}{x}, \frac{k+1/2}{x}]} \frac{dt}{t^2} \leq \frac{\theta}{2^r}.
\end{equation}
The interval 
$[2^r,2^{r+1}] \cap [\frac{k-\theta+1/2}{x}, \frac{k+1/2}{x}]$ is only non-empty when
$$ \frac{k-\theta+1/2}{x} \leq 2^{r+1}\;\text{ and }\; \frac{k+1/2}{x} \geq 2^r.$$
Hence we may restrict the $k$ summation in \eqref{k-2} to the region
$$
2^r x - \tfrac{1}{2} \leq k \leq 2^{r+1} x - \tfrac{1}{2} + \theta.
$$
By \eqref{jdef} and the fact that $0 \leq \theta < 1$, we conclude that
\begin{equation}\label{k-bound}
\frac{j}{2}-1 < k < j.
\end{equation}

We now split into cases.

{\bf Case 1:} $j=1$.  By \eqref{k-bound}, the only value of $k$ that contributes to \eqref{k-2} is $k=0$, and we can upper bound the left-hand side of \eqref{k-2} by
$$
\int_{[2^r,\frac{1/2}{x}]} \frac{dt}{t^2} = \frac{1}{2^r} - 2x$$
which is precisely $\frac{\theta}{2^r}$ as desired thanks to \eqref{jdef}.

{\bf Case 2:} $j=2$.  Now \eqref{k-bound} restricts us to $k=1$, and we can upper bound the left-hand side of \eqref{k-2} by
$$ \int_{[\frac{3/2 - \theta}{x},\frac{3/2}{x}]} \frac{dt}{t^2} = \frac{x}{3/2-\theta} - \frac{x}{3/2}$$
which by \eqref{jdef} simplifies to
$$ \frac{2}{3} \cdot \frac{2-\theta}{3-2\theta} \cdot \frac{\theta}{2^r}.$$
Since $\frac{2}{3}$ and $\frac{2-\theta}{3-2\theta}$ are at most 1,
we obtain the desired upper bound of $\frac{\theta}{2^r}$ (with a little room to spare).

{\bf Case 3:} $j=3$.  Now \eqref{k-bound} restricts us to $k=1,2$.  The left-hand side of \eqref{k-2} is now upper bounded by
$$ \int_{[2^r,\frac{3/2}{x}]} \frac{dt}{t^2} 
+ \int_{[\frac{5/2 - \theta}{x},\frac{5/2}{x}]} \frac{dt}{t^2}
= \frac{1}{2^r} - \frac{x}{3/2} + \frac{x}{5/2-\theta} - \frac{x}{5/2}$$
which by \eqref{jdef} simplifies to
$$ \biggl( \frac{1}{3} + \frac{2}{5} \cdot \frac{3-\theta}{5-2\theta}\biggr)
\frac{\theta}{2^r}.$$
Since $0 \leq \theta < 1$, we have $\frac{1}{3} + \frac{2}{5} \cdot \frac{3-\theta}{5-2\theta}
< \frac{1}{3} + \frac{2}{5} \cdot \frac{2}{3} < 1$,
and we obtain the desired upper bound of $\frac{\theta}{2^r}$ (with a little more room to spare).

{\bf Case 4:} $j=4$.  Now \eqref{k-bound} restricts us to $k=2,3$.  The left-hand side of \eqref{k-2} is now upper bounded by
$$ \int_{[\frac{5/2 - \theta}{x},\frac{5/2}{x}]} \frac{dt}{t^2}
+ \int_{[\frac{7/2 - \theta}{x},\frac{7/2}{x}]} \frac{dt}{t^2}
 = \frac{x}{5/2-\theta} - \frac{x}{5/2} + \frac{x}{7/2-\theta} - \frac{x}{7/2}$$
which by \eqref{jdef} simplifies to
$$ \biggl( \frac{2}{5} \cdot \frac{4-\theta}{5-2\theta} + \frac{2}{7} \cdot \frac{4-\theta}{7-2\theta}\biggr) \frac{\theta}{2^r}.$$
Since $0 \leq \theta < 1$, we have
$$ \frac{2}{5} \cdot \frac{4-\theta}{5-2\theta} + \frac{2}{7} \cdot \frac{4-\theta}{7-2\theta}
< \frac{2}{5} \cdot \frac{3}{3} + \frac{2}{7} \cdot \frac{3}{5} < 1$$
and we again obtain the desired upper bound of $\frac{\theta}{2^r}$ (with a fair amount\footnote{Note that the increasing ease of proof of \eqref{3r} as $j$ increases is consistent with the behavior exhibited in Figure \ref{fig:combination}.} of room to spare).  

{\bf Case 5:} $j>4$.  There is a finite interval $[k_1,k_2]$ of integers $k$ for which $[2^r,m] \cap [\frac{k-\theta+1/2}{x}, \frac{k+1/2}{x}]$ is non-empty.  From the decreasing nature of $\frac{1}{t^2}$, we have
$$ \int_{[2^r,2^{r+1}] \cap [\frac{k-\theta+1/2}{x}, \frac{k+1/2}{x}]} \frac{dt}{t^2} \leq
\theta \int_{[2^r,2^{r+1}] \cap [\frac{k-1+1/2}{x}, \frac{k+1/2}{x}]} \frac{dt}{t^2}$$
when $k_1 < k \leq k_2$.  Thus we may bound the left-hand side of \eqref{k-2} by
$$ \int_{[2^r,2^{r+1}] \cap [\frac{k_1-\theta+1/2}{x}, \frac{k_1+1/2}{x}]} \frac{dt}{t^2} + \theta \int_{[2^r,2^{r+1}]} \frac{dt}{t^2}.$$
For the first integral, we observe that the domain is an interval of length at most $\theta/x$ and the integrand is at most $1/2^{2r}$; meanwhile, the second integral can be evaluated as $\frac{1}{2^r} - \frac{1}{2^{r+1}}$.  Putting all this together, we have upper bounded the left-hand side of \eqref{k-2} by
$$ \frac{\theta}{2^{2r} x} + \frac{\theta}{2^r} - \frac{\theta}{2^{r+1}}.$$
But since $j>4$, we have from \eqref{jdef} that
$$ \frac{\theta}{2^{2r} x} \leq \frac{\theta}{2^{2r} \cdot 4/2^{r+1}} = \frac{\theta}{2^{r+1}}$$
and the claim \eqref{k-2} follows.  This concludes the proof of (i).

Now we prove (ii).  Integrating (i) against an arbitrary probability measure $\mu$ on $\R$, we conclude that
\begin{multline*}
2 \E_\mu g(x)  + \sum_{j=1}^r \E_\mu g(2^j x)\\
+  \sum_{\ell = 2^r+1}^m \frac{m}{l(l-1)} \biggl(\E_\mu g(\ell x) + \E_\mu g((2^{r+1}+1-\ell)x)- \E_\mu g(x)\biggr) \leq 1.
\end{multline*}
Since 
\begin{equation}\label{asum}
\sum_{\ell=2^r+1}^m \frac{m}{l(l-1)} = \frac{m}{2^r} - 1 < 2,
\end{equation}
the net coefficient of $\E_\mu g$ here is positive.  Thus we have
$$ \biggl(2 + \sum_{j=1}^r 1
+ \sum_{\ell = 2^r+1}^m \frac{m}{l(l-1)}(1+1-1)\biggr)
\min_{1 \leq k \leq m} \E_\mu g(kx) \leq 1$$
which gives the upper bound in \eqref{4r} after a brief computation using \eqref{asum}.

To establish the matching lower bound using the measure $\mu_m$, it suffices to show that
$$ \sum_{j=1}^{r+1} g\biggl( \frac{k}{2^j}\biggr) + \frac{m}{2^r} g\biggl( \frac{k}{2m}\biggr) = 1$$
for all $k=1,\dots,m$.  But from \eqref{tele} and telescoping series we have
$$ \sum_{j=1}^{r+1} g\biggl( \frac{k}{2^j}\biggr) = \{k\}^* - \biggl\{\frac{k}{2^{r+1}}\biggr\}^* = 1 - \frac{k}{2^{r+1}}$$
while since $\frac{k}{2m} \leq \frac{1}{2}$ we have
$$ g\biggl(\frac{k}{2m}\biggr)  = \frac{k}{2m}$$
and the claim follows.
\end{proof}

\begin{remark}
\label{non-unique}
The particular linear combination of the $g(kx)$ used in \eqref{3r} was discovered after some numerical experimentation, guided by the fact that this combination should attain the bound of $1$ at every point in the support of the optimal measure $\mu_m$.  However, this constraint does not fully determine the coefficients of the combination,
and it would be possible to establish the bound \eqref{4r} using other linear combinations of $g(kx)$
instead.  For instance, when $m$ is a power of two, the inequalities \eqref{3r} and \eqref{3r-easy} differ, even though they both imply \eqref{4r}: see Figures \ref{fig:combination2}, \ref{fig:combination3}.
\end{remark}

\begin{figure} [t]
\centering
\includegraphics[width=4in]{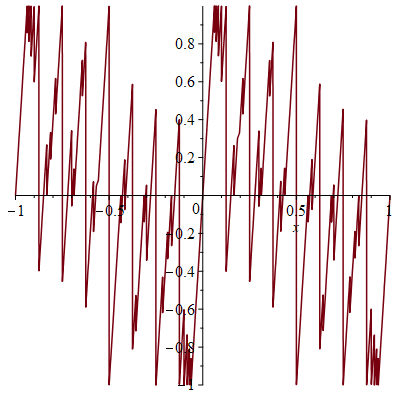}
\caption{The left-hand side of \eqref{3r} when $m=8$; compare with Figure \ref{fig:combination2}.}
\label{fig:combination3}
\end{figure}

\section{Optimal equidistribution for the sine function}
\label{sectionsine}

The function $\sin(2\pi x)$ is somewhat analogous
to the sawtooth function $g(x)$ studied in Theorem \ref{sawtooth}.
We now solve exactly the corresponding optimization problem
for the sine function in Theorem \ref{sine}.
The problem solved here
is equivalent to finding
the optimal inequality of the form $\sum_{k=1}^m a_k\sin kx\leq 1$
for each positive integer $m$.

\begin{theorem}
\label{sine}
Let $m$ be a positive integer.
\begin{itemize}
\item[(i)]  We have
\begin{equation}\label{sine-ineq}
\sum_{1\leq k\leq m; \; k\text{ odd}}\frac{2}{(m+1)^2}
\cot\biggl( \frac{\pi k}{2m+2}\biggr)\bigl[(m+1-k) \sin kx+ k \sin (m+1-k)x\bigr]\leq 1
\end{equation}
for all real numbers $x$. Write this inequality
as $\sum_{k=1}^m a_k\sin kx\leq 1$; then all the coefficients
$a_k$ are nonnegative.
\item[(ii)] We have
\begin{equation}\label{sine-r}
\min_{1 \leq k \leq m} \E_\mu \sin k x  \leq \frac{1}{c_m}
\end{equation}
for every (Borel) probability measure $\mu$ on the real line, where
\begin{equation}\label{cm-def}
c_m \coloneqq \frac{2}{m+1} \sum_{1 \leq j \leq m; \; j \text{ odd}}
\cot \biggl(\frac{\pi j}{2m+2}\biggr).
\end{equation}
This is the optimal bound: equality is attained for the measure $\mu_m$
with mass $\frac{2}{(m+1)c_m} \cot \bigl(\frac{\pi j}{2m+2}\bigr)$ at $\frac{\pi j}{m+1}$ for every odd $1 \leq j \leq m$.
\end{itemize}

\end{theorem}

It follows from our argument (in particular the properties
of the measure $\mu_m$ in (ii))
that statement (i) is the optimal inequality of the form
$\sum_{k=1}^m a_k \sin kx\leq 1$, in the sense that it has
the maximal value of $\sum a_k$ (namely, $c_m$)
among all inequalities of this form. The closest relative
of this inequality in the literature seems to be Vaaler's inequality,
of the form
$\sum_{k=1}^m a_k\sin 2\pi kx\geq x-\tfrac{1}{2}$
for $x\in [0,\tfrac{1}{2}]$
\cite[Theorem 18]{Vaaler}.

By comparing
the cotangent sum $c_m$ to an integral, with the first few
terms of the sum separated off for greater accuracy,
one checks that $c_m$ is close to $(2/\pi)\log(m+1)$.
Precisely, by an argument due to Pinelis \cite{Pinelis},
$$c_m=\frac{2}{\pi}\log(m+1)
+\frac{2}{\pi}\biggl(\log\biggl(\frac{4}{\pi}\biggr)+\gamma\biggr)+o(1),$$
where $\gamma$ is Euler's constant.
In particular, the bounds for the sine problem are of the same order of magnitude as the bounds for the sawtooth problem.  If one replaces sine by cosine then the problem becomes trivial, as in this case the Dirac mass at the origin is clearly the extremizing measure and there is no decay in $m$.

Another formula for the constant $a_k$
in inequality \eqref{sine-ineq} is
$$a_k=\frac{4(m+1-k)}{(m+1)^2} \;
\frac{\sin^2 \frac{\lceil m/2 \rceil k \pi}{m+1}}{\sin \frac{\pi k}{m+1}}$$
for $1\leq k\leq m$. We will not use this, however.

\begin{figure} [t]
\centering
\includegraphics[width=4in]{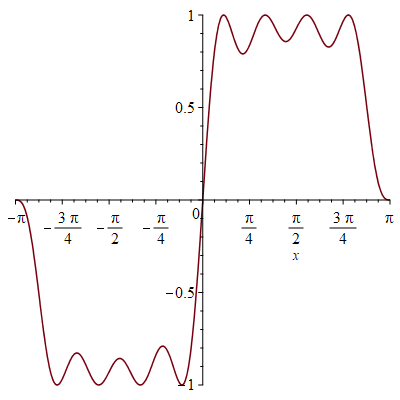}
\caption{The left-hand side $f(x)$ of \eqref{sine-ineq} when $m=8$.  We are showing that it is at most $1$.  The bound of $1$ is attained exactly at the support $\{\pi/9,3\pi/9,5\pi/9,7 \pi/9\}$ of the optimal measure $\mu_8$.}
\label{fig:combination8-sine}
\end{figure}

\begin{figure} [t]
\centering
\includegraphics[width=4in]{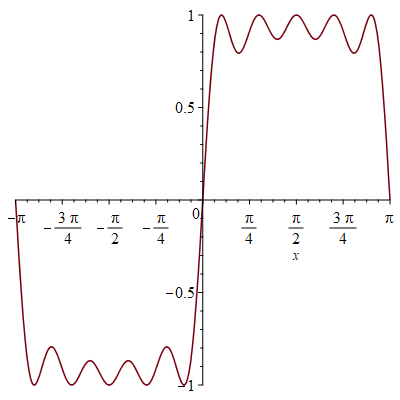}
\caption{The left-hand side $f(x)$ of \eqref{sine-ineq} when $m=9$.  We are showing that it is at most $1$.  The bound of $1$ is attained exactly at the support $\{\pi/10,3\pi/10,5\pi/10,7 \pi/10, 9\pi/10\}$ of the optimal measure $\mu_9$.}
\label{fig:combination9-sine}
\end{figure}

\begin{proof}  We first show that part (ii) follows from (i).
All the coefficients in the inequality
$\sum_{k=1}^m a_k\sin kx\leq 1$ from \eqref{sine-ineq} are nonnegative,
since $\cot x$ decreases from $\infty$ to 0 on the interval
$(0,\pi/2]$.
Therefore, on integrating this inequality
against any Borel probability measure $\mu$ on the real line, we have
$$\sum_{k=1}^m a_k\; \E_\mu \sin kx \leq 1.$$
From the definition of $a_k$ in terms of cotangents, it is immediate
that $\sum_{k=1}^m a_k = c_m$. This proves \eqref{sine-r},
namely that $\min_{1 \leq k \leq m} \E_\mu \sin k x  \leq 1/c_m$.

Next, we show that for every $1 \leq k \leq m$,
the measure $\mu_m$ defined in (ii) satisfies
$$\E_{\mu_m} \sin k x = \frac{1}{c_m}.$$
The proof seems easiest in terms
of the Fourier transform on the cyclic group $G=\Z/(2m+2)$.
Namely, for a complex-valued function $f$ on $G$, define the Fourier
transform on the dual group $\widehat{G}=\Z/(2m+2)$ by
$\widehat{f}(\xi)=(1/(2m+2))\sum_{x\in G}f(x)e^{-2\pi i \xi x/(2m+2)}$;
then the inverse Fourier transform gives that $f(x)=
\sum_{\xi\in \widehat{G}} \widehat{f}(\xi)e^{2\pi i\xi x/(2m+2)}$.

Let $\sigma_m$ be the function on $G=\Z/(2m+2)$
defined by $\sigma_m(k)=1$ if $1\leq k\leq m$, $\sigma_m(k)=-1$ if $-m\leq k\leq -1$,
and zero otherwise (a discrete version of a square wave).
Let $\zeta=e^{2\pi i/(2m+2)}$. Then
the Fourier transform of $\sigma_m$ is, for $j\in \Z/(2m+2)$,
\begin{align*}
\widehat{\sigma_m}(j)&=\frac{1}{2m+2}\sum_{k=0}^{2m+1}\sigma_m(k) \zeta^{-jk}\\
&=\frac{-1}{2m+2}\sum_{k=1}^m (\zeta^{jk}-\zeta^{-jk}).
\end{align*}
Clearly $\widehat{\sigma_m}(0)=0$. For $j\neq 0$ in $\Z/(2m+2)$, we have
$\sum_{k=1}^m \zeta^{jk}=(\zeta^j-(\zeta^j)^{m+1})/(1-\zeta^j)$.
Since $\zeta^{m+1}=-1$, that sum is $(1+\zeta^j)/(1-\zeta^j)$ if $j$
is odd and $-1$ if $j$ is even. Likewise, 
$\sum_{k=1}^m \zeta^{-jk}$ is $-(1+\zeta^j)/(1-\zeta^j)$ if $j$ is
odd and $-1$ if $j$ is even. So the Fourier transform above is given
by $\widehat{\sigma_m}(j)=(-1/(m+1))(1+\zeta^j)/(1-\zeta^j)$ if $j$ is odd
and 0 if $j$ is even. Equivalently,
$\widehat{\sigma_m}(j)=(-i/(m+1))\cot(\pi j/(2m+2))$
for $j$ odd.

Therefore, applying the inverse Fourier transform tells us,
in particular for $1\leq k\leq m$, that
\begin{align*}
1&=\sigma_m(k)\\
&=\sum_{j=0}^{2m+1}\widehat{\sigma_m}(j) e^{2\pi i jk/(2m+2)}\\
&=\sum_{1\leq j\leq m; \; j\text{ odd}}\frac{2}{m+1} \;
\cot\biggl( \frac{\pi j}{2m+2}\biggr) \;
\sin\biggl(\frac{\pi jk}{m+1}\biggr).
\end{align*}
(This can also be deduced from an identity due to Eisenstein
and Stern, discussed in the introduction to \cite{BY}.)
After dividing by $c_m$, this says that the measure $\mu_m$
defined in (ii) has $\E_{\mu_m}\sin kx=1/c_m$
for all $1\leq k\leq m$, as we want.

It remains to prove part (i).
We can relate the linear combination of sines $f(x)$
on the left side of \eqref{sine-ineq} to the function
$\sigma_m$ above.
First, let $\nu_m$ be the function on $\Z/(2m+2)$
defined by $\nu_m(j)=2\sigma_m(j)$ if $j$ is odd and 0 otherwise;
so $\nu_m(j)$ is 2 for $1\leq j\leq m$ and $j$ odd,
$-2$ for $m+2\leq j\leq 2m-1$ and $j$ odd, and 0 otherwise.
One checks that multiplying a function $f(j)$ on $\Z/(2m+2)$
by $(-1)^j$ corresponds to shifting its Fourier transform
by $m+1$. Therefore, the Fourier transform of $\nu_m$ is
$$\widehat{\nu_m}(j)=\widehat{\sigma_m}(j)-\widehat{\sigma_m}(m+1+j).$$
Since $\sigma_m$ is an odd function, so is $\widehat{\sigma_m}$,
and hence we
can rewrite this formula as
$$\widehat{\nu_m}(j)=\widehat{\sigma_m}(j)+\widehat{\sigma_m}(m+1-j).$$
If $m$ is odd, then this is
$$\widehat{\nu_m}(j)=\frac{-i}{m+1}\biggl(\cot\frac{\pi j}{2m+2}
+\cot\frac{\pi (m+1-j)}{2m+2}\biggr) $$
if $j$ is odd and 0 if $j$ is even.
If $m$ is even, then
$$\widehat{\nu_m}(j)=\begin{cases}
\frac{-i}{m+1}\cot\bigl(\frac{\pi j}{2m+2}\bigr)
&\text{if }j\text{ is odd,}\\
\frac{-i}{m+1}\cot\bigl(\frac{\pi (m+1-j)}{2m+2}\bigr)
&\text{if }j\text{ is even.}
\end{cases}$$

This is clearly related to the function $f(x)$. To make
the connection precise, consider another interpretation
of the Fourier transform on $\Z/(2m+2)$: namely, as Fourier
series on the circle $\R/2\pi\Z$ applied to linear combinations
of Dirac delta functions
with support in $(1/(2m+2))2\pi\Z$. 
Let $S$ denote the square wave function
$$ S(x) \coloneqq 1_{0 < \{x/2\pi\} < 1/2} - 1_{1/2 < \{x/2\pi\} < 1}.$$
We sample this function at odd multiples of $\pi/(m+1)$
to create a discrete approximation $\nu_m$ to $S(x)$,
basically a Dirac comb modulated by a square wave:
\begin{equation}\label{num}
 \nu_m \coloneqq \frac{2\pi}{m+1} \sum_{j \in \Z; j \text{ odd}} S\biggl(\frac{\pi j}{m+1}\biggr) \delta_{\pi j/(m+1)}.
\end{equation}
This measure is $2\pi$-periodic. Here we multiplied the function
$\nu_m$ defined earlier on $\Z/(2m+2)$ by $2\pi/(2m+2)$ in order to make
the Fourier coefficients
$$ \widehat{\nu_m}(k) \coloneqq \frac{1}{2\pi}
\int_0^{2\pi} e^{-i kx}\ d\nu(x)$$
the same as those we computed
on $\Z/(2m+2)$. (In particular,
these Fourier coefficients are periodic with period $2m+2$.)

By inspection, then, the function $f(x)$ from part (i) of the theorem is
$$ f(x) =  \frac{2i}{m+1} \sum_{1 \leq k \leq m} (m+1-k)
\widehat{\nu_m}(k) \sin kx,$$
or equivalently (due to the odd nature of $\nu_m$ and hence $\widehat{\nu_m}$)
$$ f(x) = \sum_{-m \leq k \leq m} \biggl(1 - \frac{|k|}{m+1}\biggr)
\widehat{\nu_m}(k) e^{ikx}.$$

This is clearly related to the {\it Fej\'er kernel},
$F_{m+1}(x) \coloneqq \frac{1}{m+1} \left( \frac{\sin (m+1)(x/2)}{\sin x/2}
\right)^2
=\sum_{|k|\leq m}(1-\frac{|k|}{m+1})e^{ikx}$.
(The original motivation for the Fej\'er kernel was to show
that every continuous function on the circle
is the uniform limit of the averaged partial sums of its Fourier series
\cite[Theorem 1.3.3]{DM}.)
Namely, since the Fourier series takes convolution
to multiplication,
$f$ is the convolution of $\nu_m$ with the Fej\'er kernel:
$$ f(x) = \frac{1}{2\pi} \int_0^{2\pi} F_{m+1}(x-y) \; d\nu_m(y).$$
Here $F_{m+1}(x)$ equals $m+1$ at multiples of $2\pi$,
and it vanishes at other even multiples of $\pi / (m+1)$.
So from \eqref{num} we have
$$ f\biggl(\frac{\pi j}{m+1}\biggr) = S\biggl(\frac{\pi j}{m+1}\biggr)$$
whenever $j$ is odd.  This proves:
\begin{equation}\label{fa}
 f\biggl(\frac{\pi j}{m+1}\biggr) = 1 
\end{equation}
whenever $1 \leq j \leq m$ is odd.  More broadly, we now have
an interpretation of $f$ as an interpolation of $S$
given by the Fej\'er kernel.

\begin{figure} [t]
\centering
\includegraphics[width=4in]{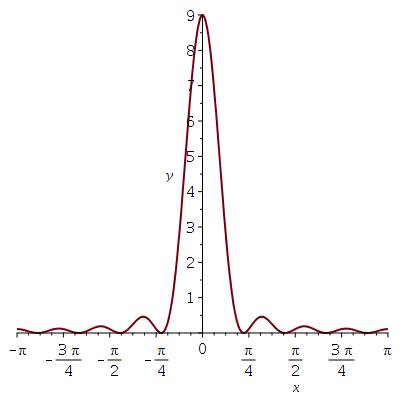}
\caption{The Fej\'er kernel $F_{m+1}(x)$ with $m=8$.}
\label{fig:fejer-8}
\end{figure}

Next, let us show that
\begin{equation}\label{fb}
 f'\biggl(\frac{\pi j}{m+1}\biggr) = 0 
\end{equation}
for any odd integer $j$.
Namely, since $f$ is the convolution of $\nu$ with the Fej\'er kernel
$F_{m+1}$, $f'$ is the convolution of $\nu$ with the derivative
$F_{m+1}'$.
But $F_{m+1}'(x)$ vanishes at all even multiples of $\pi/(m+1)$,
since $F_{m+1}$ reaches its minimum 0 or maximum $m+1$ at those points.
Since $\nu$ is supported on odd multiples of $\pi/(m+1)$, this proves
\eqref{fb}.

Clearly $f$ is odd and $2\pi$-periodic; in particular 
\begin{equation}\label{fosc}
f(0) = f(\pi) = 0.
\end{equation}
Next, since $f'(x)$ is a linear combination of the functions
$\cos kx$ for $1\leq k \leq m$,
de Moivre's formula gives that $f'(x) = P(\cos x)$ for some polynomial $P$ of degree $m$.  In particular, $f'$ has at most $m$ zeros in the interval $[0,\pi]$ (counting multiplicity for the zeros in $(0,\pi)$).  On the other hand, from \eqref{fb} we see that $f'$ has $\lfloor m/2 \rfloor + 1$ zeros in this region, at the points $\frac{\pi j}{m+1}$ with $1 \leq j \leq m+1$ odd.  On the other hand, from \eqref{fa} and Rolle's theorem (and the fact that $f$ is not locally constant) we also see that we have $\lceil m/2 \rceil - 1$ additional zeros distinct from the preceding ones, with one additional zero strictly between $\frac{\pi j}{m+1}$ and $\frac{\pi (j+2)}{m+1}$ whenever $1 \leq j \leq j+2 \leq m$ is odd.  Thus all the $m$ zeros of $f'$ are accounted for, and there are no further zeros; in particular all the zeros of $f'$ in $(0,\pi)$ are simple, and $f'$ changes sign as it crosses each zero in $(0,\pi)$.  We then conclude from \eqref{fosc}, \eqref{fa}, and the mean value theorem that
\begin{itemize}
\item $f(x)$ is strictly increasing from $0$ to $1$ as $x$ goes from $0$ to $\frac{\pi}{m+1}$;
\item Whenever $1 \leq j \leq j+2 \leq m$ is odd, $f$ starts at a local maximum of $1$ at $x = \frac{\pi j}{m+1}$, strictly decreases to a local minimum somewhere between $\frac{\pi j}{m+1}$ and $\frac{\pi (j+2)}{m+1}$, then strictly increases back to a local maximum of $1$ at $x = \frac{\pi (j+2)}{m+1}$.
\item If $j = 2 \lfloor \frac{m}{2} \rfloor + 1$ is the largest odd number less than $m+1$, $f$ is strictly decreasing from $1$ to $0$ as $x$ goes from $\frac{\pi j}{m+1}$ to $\pi$.
\end{itemize}
(See Figures \ref{fig:combination8-sine}, \ref{fig:combination9-sine}.)
This already establishes that $f(x) \leq 1$ when $0 \leq x \leq \pi$.
If we can also show that $f(x) \geq -1$ for $0 \leq x \leq \pi$, then as $f$ is odd and $2\pi$-periodic we will have $f(x) \leq 1$ for all $x$,
proving the theorem. In reality, this bound will be true with substantial room to spare, as $f(x)$ is only moderately smaller than $1$ on most of the interval
$[0,\pi]$ (cf.\ the Gibbs phenomenon).

From the observations above and the oddness of $f$,
we know that $f(x)\geq -1$ if $x\in [-\pi/(m+1),\pi/(m+1)]$
and also if $x\in [m\pi/(m+1),(m+2)\pi/(m+1)]$. So it suffices
to show that $f(x+\frac{\pi}{m+1})-f(x-\frac{\pi}{m+1})]$ is nonnegative for
$x\in [0, (m-1)\pi/(2m+2)]$, and nonpositive for $x\in [(m+3)\pi/(2m+2),
\pi]$.

Assume that $m$ is odd.
Since $f$ is the convolution of $\nu_m$ with the Fej\'er kernel
$F_{m+1}$, we have
\begin{align*}
f\biggl(x+\frac{\pi}{m+1}\biggr) -f\biggl(x-\frac{\pi}{m+1}\biggr)
&= \biggl[\nu_m\biggl( x+\frac{\pi}{m+1}\biggr)
-\nu_m\biggl(x-\frac{\pi}{m+1}\biggr)\biggr]*F_{m+1}\\
&= \frac{4\pi}{m+1}\bigl[ \delta_0-\delta_{\pi}\bigr] *F_{m+1}\\
&= \frac{4\pi}{m+1}\bigl[ F_{m+1}(x)-F_{m+1}(x-\pi)\bigr].
\end{align*}
(This description of the change of $\nu_m$ when the input changes
by $2\pi/(m+1)$ uses that $m$ is odd.)
So it suffices
to show that $F_{m+1}(x)-F_{m+1}(x-\pi)$ is nonnegative for
$x\in [0, (m-1)\pi/(2m+2)]$, and that it is nonpositive
for $x\in [(m+3)\pi/(2m+2),
\pi]$. We prove this (in fact for slightly bigger intervals)
in Lemma \ref{fejer} below.
That completes the proof for $m$ odd. 

For $m$ even,
we have the somewhat messier situation that
\begin{align*}
{} &\; f\biggl(x+\frac{\pi}{m+1}\biggr) -f\biggl(x-\frac{\pi}{m+1}\biggr)\\
{}=&\; \biggl[\nu_m\biggl( x+\frac{\pi}{m+1}\biggr)
-\nu_m\biggl(x-\frac{\pi}{m+1}\biggr)\biggr]*F_{m+1}\\
{}=&\; \frac{2\pi}{m+1}\bigl[ 2\delta_0-\delta_{m\pi/(m+1)}
-\delta_{(m+2)\pi/(m+1)}\bigr] *F_{m+1}\\
{}=&\; \frac{2\pi}{m+1}\biggl[2F_{m+1}(x)
-F_{m+1}\biggl(x-\frac{m\pi}{m+1}\biggr)
-F_{m+1}\biggl(x-\frac{(m+2)\pi}{m+1}\biggr)\biggr].
\end{align*}
This is nonnegative for $x\in [0, \frac{(m-1)\pi}{2m+2}]$
and nonpositive for $x\in [\frac{(m+3)\pi}{2m+2},\pi]$, by Lemma
\ref{fejer} below (which in fact works for slightly bigger
intervals). This completes the proof of Theorem
\ref{sine}.
\end{proof}

\begin{lemma}
\label{fejer}
Let $m$ be a positive integer, and let $F(x)=F_{m+1}(x)
=\frac{1}{m+1} \left( \frac{\sin (m+1)(x/2)}{\sin x/2} \right)^2$
be the Fej\'er kernel.
\begin{itemize}
\item[(i)] If $m$ is odd, then $F(x)-F(x-\pi)$
is nonnegative if $x\in [-\pi/2,\pi/2]$ and nonpositive
if $x\in [\pi/2,3\pi/2]$.
\item[(ii)] If $m$ is even,  then $2F(x)-F(x-\frac{m\pi}{m+1})
-F(x-\frac{(m+2)\pi}{m+1})$ is nonnegative if $x\in [-\frac{\pi}{2}
+\frac{\pi}{m+1},\frac{\pi}{2}-\frac{\pi}{m+1}]$
and nonpositive if $x\in [\frac{\pi}{2}+\frac{\pi}{m+1},
\frac{3\pi}{2}-\frac{\pi}{m+1}]$.
\end{itemize}
\end{lemma}

\begin{proof}
(i) Let $m$ be an odd positive integer. By definition of
the Fej\'er kernel $F(x)$, we have
$$F(x)-F(x-\pi)=\frac{1}{m+1}
\biggl[ \frac{\sin^2 ((m+1)x/2)}{\sin^2 x/2}
-\frac{\sin^2 (\frac{(m+1)x}{2}-\frac{(m+1)\pi}{2})}
{\sin^2 (x-\pi)/2}\biggr].$$
Here $\sin((x-\pi)/2)=-\cos (x/2)$. Since $m$ is odd,
$(m+1)\pi/2$ is an integer multiple of $\pi$, and so the numerators
of the two terms are equal. We deduce that
$$F(x)-F(x-\pi)=\frac{\sin^2((m+1)x/2)}{m+1}
\biggl[ \frac{1}{\sin^2(x/2)}-\frac{1}{\cos^2(x/2)}\biggr].$$
For $x\in [-\pi/2,\pi/2]$, we have $\sin^2 (x/2)\leq 1/2$, while
$\cos^2(x/2)\geq 1/2$.
It follows that $F(x)-F(x-\pi)$
is nonnegative if $x\in [-\pi/2,\pi/2]$.
By applying that result to $x-\pi$ in place of $x$, we also find
that $F(x)-F(x-\pi)$ is nonpositive if $x\in [\pi/2,
3\pi/2]$, as we want.

(ii) Let $m$ be an even positive integer. By definition
of the Fej\'er kernel $F(x)$,
\begin{align*}
{}&\; 2F(x)-F(x-m\pi/(m+1))
-F(x-(m+2)\pi/(m+1))\\
{}=&\; \frac{1}{m+1}\biggl[\frac{2\sin^2 ((m+1)x/2)}{\sin^2 (x/2)}
-\frac{\sin^2 (\frac{(m+1)x}{2}-\frac{m\pi}{2})}{\sin^2 (\frac{x}{2}
-\frac{m\pi}{2m+2})}
-\frac{\sin^2 (\frac{(m+1)x}{2}-\frac{(m+2)\pi}{2})}{\sin^2 (\frac{x}{2}
-\frac{(m+2)\pi}{2m+2})}
\biggr].
\end{align*}
Since $m$ is even, $m\pi/2$ and $(m+2)\pi/2$ are both integer
multiples of $\pi$, and so the three $\sin^2$ terms
in the numerators
are equal. So we can rewrite the expression above as
$$\frac{\sin^2((m+1)x/2)}{m+1}\biggl[\frac{2}{\sin^2 (x/2)}
-\frac{1}{\cos^2(\frac{x}{2}+\frac{\pi}{2m+2})}
-\frac{1}{\cos^2(\frac{x}{2}-\frac{\pi}{2m+2})}\biggr].$$

For $x\in [-\frac{\pi}{2}+\frac{\pi}{m+1},\frac{\pi}{2}-\frac{\pi}{m+1}]$,
we have $\sin^2(x/2)\leq 1/2$,
while $\cos^2(\frac{x}{2}-\frac{\pi}{2m+2})$ and $\cos^2(\frac{x}{2}
+\frac{\pi}{2m+2})$
are both $\geq 1/2$. It follows that the previous paragraph's
expression is nonnegative for this range of $x$.
Likewise, if $x\in [\frac{\pi}{2}+\frac{\pi}{m+1},\frac{3\pi}{2}
-\frac{\pi}{m+1}]$,
then $\sin^2(x/2)\geq 1/2$
while $\cos^2(\frac{x}{2}+\frac{\pi}{2m+2})$ and $\cos^2(\frac{x}{2}
-\frac{\pi}{2m+2})$
are both $\leq 1/2$. It follows that the previous paragraph's
expression is nonpositive for this range of $x$.
Lemma \ref{fejer} is proved.
\end{proof}


\end{document}